\documentclass[12pt, twoside]{article}

\usepackage{amsmath,amsthm,amssymb}

\usepackage{times}

\usepackage{enumerate}

\makeatletter

\def\author#1{\gdef\autrun{\def\and{\unskip, }#1}\gdef\@author{#1}}

\makeatother

\def\subjclass#1{{\renewcommand{\thefootnote}{}%

\footnote{\emph{Mathematics Subject Classification (2010):} #1}}}

 \newtheorem{theorem}{\sc Theorem}[section]
 \newtheorem{lemma}[theorem]{\sc Lemma}
 \newtheorem{cor}[theorem]{\sc Corollary}
 
 \newtheorem{conj}[theorem]{\sc Conjecture}

\begin{document}

\baselineskip=17pt

\author{Simeon Ball and Jan De Beule \thanks{The first author acknowledges the support of the project MTM2008-06620-C03-01 of the Spanish Ministry of Science and Education and the project 2009-SGR-01387 of the Catalan Research Council. The second author is a postdoctoral research fellow of the Research Foundation Flanders -- Belgium (FWO)}}

\title{On sets of vectors of a finite vector space in which every subset of basis size is a basis II\footnote{final version, to appear in Des. Codes Cryptogr.}}

\date{25 January 2012}

\maketitle



\begin{abstract}
This article contains a proof of the MDS conjecture for $k \leq 2p-2$. That is, that if $S$ is a set of vectors of ${\mathbb F}_q^k$ in which every subset of $S$ of size $k$ is a basis, where $q=p^h$, $p$ is prime and $q$ is not and $k \leq 2p-2$, then $|S| \leq q+1$. It also contains a short proof of the same fact for $k\leq p$, for all $q$.
\end{abstract}

\section{Introduction}

Let $S$ be a set of vectors of ${\mathbb F}_q^k$ in which every subset of size $k$ is a basis.

In 1952, Bush \cite{Bush1952} showed that if $k \geq q$ then $|S| \leq k+1$ and the bound is attained if and only if $S$ is equivalent to $\{e_1,\ldots,e_k,e_1+\ldots+e_k\}$, where $\{e_1,\ldots,e_k\}$ is a basis.

The main conjecture for maximum distance separable codes (the MDS conjecture), proposed (as a question) by Segre \cite{Segre1955b} in 1955 is the following.

\begin{conj} \label{MDSc}
A set $S$ of vectors of the vector space ${\mathbb F}_q^k$, with the property that every subset of $S$ of size $k \leq q$ is a basis, has size at most $q+1$, unless $q$ is even and $k=3$ or $k=q-1$, in which case it has size at most $q+2$.
\end{conj}

In this article we shall prove the conjecture for all $k \leq 2p-2$, where $q=p^h$, $p$ is prime and $q$ is not prime.

We shall also prove the conjecture for $q$ prime, which was first proven in \cite{Ball2011}. It may help the reader to look at the first four sections of \cite{Ball2011}, although this article is self-contained (with the exception of the proof of Lemma~\ref{segre}) and can be read independently. The proof here is based on the ideas of \cite{Ball2011} which themselves are based on the initial idea of Segre in \cite{Segre1955a}.

For a complete list of when the conjecture is known to hold for $q$ non-prime, see \cite{HS2001} and also \cite{HT1991}. 

The best known bounds, up to first-order of magnitude ($c_i$ are constants), are that for $q$ an odd non-square, the conjecture holds for $k < \sqrt{pq}/4+c_1p$, Voloch \cite{Voloch1991}. For $q=p^{2h}$, where $p \geq 5$ is a prime, the conjecture holds for $k \leq \sqrt{q}/2+c_2$, Hirschfeld and Korchm\'aros \cite{HK1996}, and here we shall prove the conjecture for $k \leq 2\sqrt{q}+c_3$ in the case $q=p^2$. The conjecture is known to hold for all $q \leq 27$ and for all $k\leq 5$ and $k=6$ with some exceptions.

Conjecture~\ref{MDSc} has implications for various problems in combinatorics, most notably for maximum distance separable codes (whence the name) from coding theory and the uniform matroid from matroid theory. 

A {\em linear maximum distance separable code} is a linear code of length $n$, dimension $k$ and minimum distance $d$ over ${\mathbb F}_q$, for which $d=n-k+1$. Conjecture~\ref{MDSc} implies that a linear maximum distance separable code has length $n$ at most $q+1$ unless $q$ is even and $k=3$ or $k=q-1$, in which case it has length at most $q+2$. For more details on codes and MDS codes in particular, see \cite{MS1977}. 

A {\em matroid} $M=(E,F)$ is a pair in which $E$ is a set and $F$ is a set of subsets of $E$, called {\em independent sets}, such that (1) every subset of an independent set is an independent subset; and (2) for all $A \subseteq E$, all maximal independent subsets of $A$ have the same cardinality, called the {\em rank} of $A$ and denoted $r(A)$. The maximal independent sets of the {\em uniform matroid} of rank $r$ are all the $r$ element subsets of the set $E$. Conjecture~\ref{MDSc} implies that the uniform matroid of rank $r$, with $|E| \geq r+2$, is representable over ${\mathbb F}_q$ if and only if $|E| \leq q+1$, unless $q$ is even and $r=3$ or $r=q-1$, in which case it is if and only if $|E| \leq q+2$.
For more details on matroids and representations of matroids in particular, see \cite{Oxley1992}.

\section{The tangent function and the Segre product} \label{lemmas}

For any subset $Y$ of $k-2$ elements of $S$, since there are at most $k-1$ vectors of $S$ in a hyperplane, there are exactly 
$$t=q+1-(|S|-k+2)=q+k-1-|S|$$
hyperplanes containing $Y$ and no other vector of $S$. 

We shall assume throughout that $t \geq 1$, which is no restriction since we are trying to prove $|S| \leq q+1$ for $k\geq 4$.

Let $\phi_Y$ be a set of $t$ linearly independent linear maps from ${\mathbb F}_q^k$ to ${\mathbb F}_q$ with the property that for each $\alpha \in \phi_Y$, $Ker(\alpha)$ is one of the $t$ hyperplanes containing $Y$ and no other vector of $S$.

The {\em tangent function at $Y$} is defined (up to scalar factor) as
$$
T_Y(x) =\prod_{\alpha \in \phi_Y} \alpha(x),
$$
and is a map from ${\mathbb F}_q^k$ to ${\mathbb F}_q$.

The following is a coordinate-free version of Segre's lemma of tangents \cite{Segre1967} and is from \cite{Ball2011}.

\begin{lemma}  \label{segre}
Let $D$ be a set of $k-3$ elements of $S$. For all $x,y,z \in S \setminus D$
$$
T_{\{ x \} \cup D}(y)T_{\{ y \} \cup D}(z)T_{\{ z\}\cup D}(x)=(-1)^{t+1}T_{\{ x\}\cup D}(z)T_{\{ y\}\cup D}(x)T_{\{ z\}\cup D}(y).
$$
\end{lemma}

Since we wish to write $\det(A)$ where $A=\{a_1,\ldots,a_k\}$ is a subset of $S$, to mean the determinant $\det(a_1,\ldots,a_k)$, we order the elements of $S$ from now on. We write $\det(A_1,\ldots,A_r)$ to mean the determinant in which the elements of $A_1$ come first, then the elements of $A_2$, etc.

The following, which follows from interpolating the tangent function, is also from \cite{Ball2011}.
\begin{lemma} \label{interpolation}
If $|S| \geq k+t>k$ then for any $Y$ of size $k-2$ and $E$ of size $t+2$, disjoint subsets of $S$, 
$$
0=\sum_{a \in E}T_{ Y}(a) \prod_{z \in E \setminus \{ a \}} \det (z,a,Y)^{-1}.$$
\end{lemma}

Let $A=( a_1,\ldots,a_n )$ and $B=( b_0,\ldots,b_{n-1} )$ be two subsequences of $S$ of the same length $n$ and let $D$ be a subset of $S\setminus (A\cup B)$ of size $k-n-1$. 

We define the {\em Segre product of $A$ and $B$ with base $D$} to be
$$
P_D(A,B)=\prod_{i=1}^n \frac{T_{D \cup \{a_1,\ldots,a_{i-1},b_i,\ldots,b_{n-1} \}} (a_i)}{T_{D \cup \{a_1,\ldots,a_{i-1},b_i,\ldots,b_{n-1} \}} (b_{i-1})}
$$
and $P_D(\emptyset,\emptyset)=1$.

The following lemmas are a consequence of Lemma~\ref{segre}.

\begin{lemma} \label{numerator}
$$P_D(A^*,B)=(-1)^{t+1}P_D(A,B),$$
where the sequence $A^*$ is obtained from $A$ by interchanging two elements.
\end{lemma}

\begin{proof}
It is enough to prove the lemma for two adjacent elements in $A$ since the transposition $(j \ \ell)$ can be written as the product of $2(\ell-j)+1$ transpositions of the form $(n \ n+1)$. 

The only terms in the Segre product which differ when we interchange $a_j$ and $a_{j+1}$ are the terms in the product for $i=j$ and $i=j+1$. Trivially
$$
\frac{T_{D \cup \{a_1,\ldots,a_{j-1},b_j,b_{j+1},\ldots,b_{n-1} \}} (a_j)}{T_{D \cup \{a_1,\ldots,a_{j-1},b_j,b_{j+1},\ldots,b_{n-1} \}} (b_{j-1})}
\frac{T_{D \cup \{a_1,\ldots,a_{j-1},a_j,b_{j+1},\ldots,b_{n-1} \}} (a_{j+1})}{T_{D \cup \{a_1,\ldots,a_{j-1},a_j,b_{j+1},\ldots,b_{n-1} \}} (b_{j})}
$$
is equal to
$$\frac{T_{\Delta \cup \{ b_j \}}(a_j)}{T_{\Delta \cup \{ b_j \}}(b_{j-1})}\frac{T_{\Delta \cup \{ a_j \}}(a_{j+1})}{T_{\Delta \cup \{ a_j \}}(b_{j})},
$$
where $\Delta=D \cup \{a_1,\ldots,a_{j-1},b_{j+1},\ldots,b_{n-1} \}$, which is equal to
$$
(-1)^{t+1} \frac{T_{\Delta \cup \{ b_j \}}(a_{j+1})}{T_{\Delta \cup \{ b_j \}}(b_{j-1})}\frac{T_{\Delta \cup \{ a_{j+1} \}}(a_{j})}{T_{\Delta \cup \{ a_{j+1} \}}(b_{j})},
$$
by Lemma~\ref{segre}.
\end{proof}

In the same way the following lemma also holds.

\begin{lemma} \label{denominator}
$$P_D(A,B^*)=(-1)^{t+1}P_D(A,B),$$
where the sequence $B^*$ is obtained from $B$ by interchanging two elements.
\end{lemma}

The following lemma will also be needed.

\begin{lemma} \label{switch}
If $A$ and $B$ are subsequences of $S$ and $|A|=|B|-1$ then
$$
\frac{T_{D \cup B} (y)}{T_{D \cup B} (x)} P_{D \cup \{y \}} ( \{ x \} \cup A,B)=(-1)^{t+1} P_{D \cup \{x \}} ( \{ y \} \cup A,B).
$$
\end{lemma}

\begin{proof}
Using the definition of the Segre product and Lemma~\ref{segre},
$$
\frac{T_{D \cup B} (y)}{T_{D \cup B} (x)} P_{D \cup \{y \}} ( \{ x \} \cup A,B)=
\frac{T_{D \cup B} (y)}{T_{D \cup B} (x)}
\frac{T_{D \cup \{ b_1,\ldots, b_{n-1},y \} } (x)}{T_{D \cup \{ b_1,\ldots, b_{n-1},y \} } (b_0)} P_{D \cup \{ x,y \}} ( A,B \setminus \{ b_0 \} )
$$
$$
=(-1)^{t+1}\frac{T_{D \cup \{ b_1,\ldots, b_{n-1},x \} } (y)}{T_{D \cup \{ b_1,\ldots, b_{n-1},x \} } (b_0)} P_{D \cup \{ x,y \}} ( A,B \setminus \{ b_0 \} )=(-1)^{t+1} P_{D \cup \{x \}} ( \{ y \} \cup A,B).
$$
\end{proof}

\section{The main lemma}

For any subset $B$ of an ordered set $L$, let $\sigma(B,L)$ be $(t+1)$ times the number of transpositions needed to order $L$ so that the elements of $B$ are the last $|B|$ elements.

\begin{lemma}  \label{basic}
Let $A$ of size $n$, $L$ of size $r$, $D$ of size $k-1-r$ and $\Omega$ of size $t+1-n$ be  pairwise disjoint subsequences of $S$. If $n \leq r \leq n+p-1$ and $r\leq t+2$, where $q=p^h$, then
$$
\sum_{\substack{B \subseteq L \\ |B|=n}} (-1)^{\sigma(B,L)} P_{D\cup (L \setminus B)}(A,B) \prod_{z \in \Omega \cup B} \det(z,A,L\setminus B,D)^{-1}=
$$
$$
(-1)^{(r-n)(nt+n+1)} \sum_{\substack{\Delta \subseteq \Omega \\ |\Delta|=r-n}}  P_D(A \cup \Delta,L) \prod_{z \in (\Omega \setminus \Delta)\cup L} \det(z,A,\Delta,D)^{-1}.
$$
\end{lemma}

\begin{proof}
By induction on $r$. The case $r=n$ is straightforward. 

Fix an $x \in L$ and apply the inductive step to $L\setminus \{ x\}$ and $\{x\} \cup D$,
$$
\sum_{\substack{B \subseteq L\setminus \{ x \} \\ |B|=n}} (-1)^{\sigma(B,L \setminus \{x \})} P_{D\cup (L \setminus B)}(A,B) \prod_{z \in \Omega \cup B} \det(z,A,L\setminus (B \cup \{x\}),x,D)^{-1}=
$$
$$
(-1)^{(r-n-1)(nt+n+1)} \sum_{\substack{\Delta \subseteq \Omega \\ |\Delta|=r-n-1}}  P_{D \cup \{x\}}(A \cup \Delta,L\setminus \{x\}) \prod_{z \in (\Omega \setminus \Delta)\cup L} \det(z,A,\Delta,x,D)^{-1}.
$$
Let $\Delta$ be a subset of $\Omega$ of size $r-n-1$. The set $\Omega\setminus \Delta$ has size $t+1-n-(r-n-1)=t+2-r$ and so since $r \leq t+2$ we can apply Lemma~\ref{interpolation}, with $E=L \cup (\Omega \setminus \Delta)$ and $Y=D \cup A \cup \Delta$, and get
$$
0=\sum_{x \in L} T_{D \cup A \cup \Delta}(x) \prod_{z \in (\Omega \setminus \Delta) \cup (L \setminus \{x\})} \det (z,A,\Delta,x,D)^{-1}$$
$$
+ \sum_{y \in \Omega \setminus \Delta} T_{D \cup A \cup \Delta}(y) \prod_{z \in (\Omega \setminus (\{y\} \cup \Delta)) \cup L} \det(z,A,\Delta,y,D)^{-1}.
$$

Multiply this equation by $P_D(A \cup \Delta \cup d,L)T_{D \cup A \cup \Delta} (d)^{-1}$ for some $d$ for which $T_{D \cup A \cup \Delta}(d) \neq 0$.
By Lemma~\ref{denominator} we can rearrange $L$ so that the last element is $x$, which changes the sign by $\sigma(x,L)$. This gives
$$
0=\sum_{x \in L} (-1)^{\sigma(x,L)} P_{D \cup x}(A \cup \Delta,L\setminus \{x\}) \prod_{z \in (\Omega \setminus \Delta) \cup (L \setminus \{x\})} \det (z,A,\Delta,x,D)^{-1}+
$$
$$
 \sum_{y \in \Omega \setminus \Delta} P_{D}(A \cup \Delta \cup \{y\},L) \prod_{z \in (\Omega \setminus (\Delta \cup \{y\})) \cup L} \det(z,A,\Delta,y,D)^{-1},
$$
since 
$$
P_D(A \cup \Delta \cup \{d\},L)T_{D \cup A \cup \Delta}(x)T_{D \cup A \cup \Delta}(d)^{-1}=P_{D \cup x}(A \cup \Delta,L \setminus x)
$$
and by Lemma~\ref{switch} (and Lemma~\ref{numerator})
$$
P_D(A \cup \Delta \cup \{d\},L)T_{D \cup A \cup \Delta}(y) T_{D \cup A \cup \Delta} (d)^{-1}=P_D(A \cup \Delta \cup y,L).
$$
Note that in the second term we can order $\Delta \cup \{ y\}$ in any way we please without changing the sign since, by Lemma~\ref{numerator}, interchanging two elements of $\Delta \cup \{y\}$ in \\
$P_D(A \cup \Delta \cup \{y\},L)$ changes the sign by $(-1)^{t+1}$, exactly the same change occurs when we interchange the same vectors in the product of determinants.

Therefore, when we sum this equation over subsets $\Delta$ of $\Omega$ of size $r-n-1$ and apply the induction hypothesis, we get

$$
0=\sum_{x \in L} (-1)^{\sigma(x,L)+(r-n-1)(nt+n+1)} 
\sum_{\substack{B \subset L\setminus \{x\} \\ |B|=n}} (-1)^{\sigma(B,L \setminus \{x\})} P_{D\cup (L \setminus B)}(A,B) 
$$
$$
\prod_{z \in \Omega \cup B} \det(z,A,L\setminus (B \cup \{x\}),x, D)^{-1}+
$$
$$
(r-n)\sum_{\substack{\Delta \subseteq \Omega \\ |\Delta|=r-n}}  P_D(A \cup \Delta,L) \prod_{z \in (\Omega \setminus \Delta)\cup L} \det(z,A,\Delta,D)^{-1}.
$$

Since 
$$
\sigma(B,L)=\sigma(x,L)+\sigma(B, L \setminus \{x\})+\sigma(x,L \setminus (B \cup \{x\}))+n(t+1),
$$ 
this equation gives
$$
(-1)^{(r-n)(nt+n+1)}(r-n) \sum_{\substack{B \subset L \\ |B|=n}} (-1)^{\sigma(B,L)} P_{D\cup (L \setminus B)}(A,B) \prod_{z \in \Omega \cup B} \det(z,A,L\setminus B,D)^{-1}=
$$
$$
(r-n)\sum_{\substack{\Delta \subseteq \Omega \\ |\Delta|=r-n}}  P_D(A \cup \Delta,L) \prod_{z \in (\Omega \setminus \Delta)\cup L} \det(z,A,\Delta,D)^{-1},
$$
which is what we wanted to prove.
\end{proof}

\begin{theorem} \label{prime}
If $k \leq p$ then $|S| \leq q+1$.
\end{theorem}

\begin{proof}
If $|S|=q+2$ then $t=k-3$. 
If $q$ is prime then, by \cite[Lemma 5.1]{Ball2011}, we may dualise in ${\mathbb F}_q^{q+2}$, if necessary, to assume that $k \leq (q+1)/2$ and so $k+t \leq q+2$. 

Since $k+t \leq q+2$ we can apply Lemma~\ref{basic} with $r=t+2=k-1$ and $n=0$ and get
$$
\prod_{z \in \Omega } \det(z,L)^{-1}=0,
$$
which is a contradiction. 
\end{proof}


\section{The case $|S|=q+2$ and $q$ is non-prime.}

For any subsequence $X=\{x_1,\ldots,x_m\}$ of $S$ and $\tau \subseteq \{1,2,\ldots,m \}$, define the subsequence $X_{\tau}=\{ x_i \ | \ i \in \tau \}$.

\begin{lemma} \label{twotothen}
Suppose that $|S|=q+2$ and $n\geq k-p$. Let $A$ of size $n-m$, $L$ of size $k-1-m$, $\Omega$ of size $k-2-n$, $X$ of size $m$, $Y$ of size $m$ be disjoint subsequences of $S$. Then
$$
0=\sum_{\substack{B \subseteq L \\ |B|=n-m}} \sum_{\tau \subseteq M}(-1)^{\sigma(B,L)+\sigma(X_{\tau},X)+|\tau|} P_{(L \setminus B)\cup X_{M \setminus \tau}}(A \cup Y_{\tau},B \cup X_{\tau})
$$
$$
\times \prod_{z \in \Omega \cup B \cup X_{\tau} \cup Y_{M \setminus \tau}} \det(z,A,X_{M\setminus \tau},Y_{\tau},L\setminus B)^{-1},
$$
where $M=\{1,\ldots,m \}$.
\end{lemma}

\begin{proof}
By induction on $m$. For $m=0$ this is Lemma~\ref{basic} with $r=t+2=k-1$, which gives the bound $n \geq k-p$.

Suppose that $X$ and $Y$ have size $m$ and that $x, y \in S$ are not contained in $X$, $Y$, $L$ or $A$. We wish to prove the equation for $X \cup \{ x \}$, $Y \cup \{ y \}$, $L$ and $A$, where $|L|=k-2-m$ and $|A|=n-m-1$.

Apply the inductive step to $\{ y \} \cup L$, $A \cup \{ x \}$, $X$ and $Y$.

Writing the first sum as two sums depending on whether $B$ contains $y$ or not, we have
$$
0=\sum_{\substack{B \subseteq L \\ |B|=n-m}} \sum_{\tau \subseteq M}(-1)^{\sigma(B,L)+\sigma(X_{\tau},X)+|\tau|} P_{(L \setminus B) \cup \{ y \} \cup X_{M \setminus \tau}}(A \cup \{ x \} \cup Y_{\tau},B \cup X_{\tau})
$$
$$
\times
\prod_{z \in \Omega \cup B \cup X_{\tau} \cup Y_{M \setminus \tau}} \det(z,A,x ,X_{M\setminus \tau},Y_{\tau}, y, L\setminus B)^{-1}
$$
$$
+\sum_{\substack{B \subseteq L \\ |B|=n-m-1}} \sum_{\tau \subseteq M}(-1)^{\sigma(\{y\} \cup B,\{ y \} \cup L)+\sigma(X_{\tau},X)+|\tau|} P_{(L \setminus B)\cup X_{M \setminus \tau}}(A \cup \{ x \} \cup Y_{\tau}, \{ y \} \cup B \cup X_{\tau})
$$
$$
\times
\prod_{z \in \Omega \cup B \cup X_{\tau} \cup Y_{M \setminus \tau} \cup \{ y \}} \det(z,A,x,X_{M\setminus \tau},Y_{\tau}, L\setminus B)^{-1}.
$$
By Lemma~\ref{numerator}, then Lemma~\ref{switch} and then Lemma~\ref{numerator} again, we have
$$
\frac{T_{L \cup X}(y)}{T_{L \cup X}(x)} P_{(L \setminus B) \cup \{ y \} \cup X_{M \setminus \tau}}(A \cup \{ x \} \cup Y_{\tau},B \cup X_{\tau})=
$$
$$
(-1)^{(n-m+1)(t+1)} \frac{T_{L \cup X}(y)}{T_{L \cup X}(x)} P_{(L \setminus B) \cup \{ y \} \cup X_{M \setminus \tau}}(\{ x \} \cup A \cup Y_{\tau},B \cup X_{\tau})=
$$
$$
(-1)^{(n-m)(t+1)} P_{(L \setminus B) \cup \{ x \} \cup X_{M \setminus \tau}}(\{ y \} \cup A \cup Y_{\tau},B \cup X_{\tau})
$$
$$
=(-1)^{t+1}P_{(L \setminus B) \cup \{ x \} \cup X_{M \setminus \tau}}(A \cup \{ y \} \cup Y_{\tau},B \cup X_{\tau}),
$$
and by Lemma~\ref{numerator} and the definition of the Segre product
$$
\frac{T_{L \cup X}(y)}{T_{L \cup X}(x)} P_{(L \setminus B)\cup X_{M \setminus \tau}}(A \cup \{ x \} \cup Y_{\tau}, \{ y \} \cup B \cup X_{\tau})=
$$
$$
(-1)^{(n-m+1)(t+1)}\frac{T_{L \cup X}(y)}{T_{L \cup X}(x)} P_{(L \setminus B)\cup X_{M \setminus \tau}}(\{ x \} \cup A \cup Y_{\tau}, \{ y \} \cup B \cup X_{\tau})=
$$
$$
(-1)^{(n-m+1)(t+1)}P_{(L \setminus B)\cup X_{M \setminus \tau} \cup \{ x \} }(A \cup Y_{\tau},  B \cup X_{\tau}).
$$
Thus, multiplying the equation before by $T_{L \cup X}(y)T_{L \cup X}(x)^{-1}$ and noting that 
$$\sigma(\{y\} \cup B,\{ y \} \cup L)=\sigma(B,L)+(k-n-1)(t+1),$$
we have
$$
0=\sum_{\substack{B \subseteq L \\ |B|=n-m}} \sum_{\tau \subseteq M}(-1)^{\sigma(B,L)+\sigma(X_{\tau},X)+|\tau|} P_{(L \setminus B) \cup \{ x \} \cup X_{M \setminus \tau}}(  A  \cup \{ y \} \cup Y_{\tau},B \cup X_{\tau})
$$
$$
\times
\prod_{z \in \Omega \cup B \cup X_{\tau} \cup Y_{M \setminus \tau}} \det(z,A,x ,X_{M\setminus \tau},Y_{\tau}, y, L\setminus B)^{-1}
$$
$$
+\sum_{\substack{B \subseteq L \\ |B|=n-m-1}} \sum_{\tau \subseteq M}(-1)^{\sigma(B,L)+\sigma(X_{\tau},X)+|\tau|+(k-m-1)(t+1)} P_{(L \setminus B)\cup X_{M \setminus \tau} \cup \{ x \}}(A \cup Y_{\tau}, B \cup X_{\tau})
$$
$$
\times
\prod_{z \in \Omega \cup B \cup X_{\tau} \cup Y_{M \setminus \tau} \cup \{ y \}} \det(z,A,x,X_{M\setminus \tau},Y_{\tau}, L\setminus B)^{-1}.
$$
Applying the inductive step to $\{x \} \cup L$, $A \cup \{ y \}$, $X$ and $Y$ and writing the sum as two sums depending on whether $B$ contains $x$ or not, gives an equation similar to the above. The first sum in both equations vary only in the position of $x$ and $y$ in the determinants. Switching these in the above, multiplying by $(-1)^{t+1}$, and equating the two second sums gives,
$$
\sum_{\substack{B \subseteq L \\ |B|=n-m-1}} \sum_{\tau \subseteq M}(-1)^{\sigma(B,L)+\sigma(X_{\tau},X)+|\tau|+(k-m)(t+1)} P_{(L \setminus B)\cup X_{M \setminus \tau} \cup \{ x \}}(A \cup Y_{\tau}, B \cup X_{\tau})
$$
$$
\times
\prod_{z \in \Omega \cup B \cup X_{\tau} \cup Y_{M \setminus \tau} \cup \{ y \}} \det(z,A,x,X_{M\setminus \tau},Y_{\tau}, L\setminus B)^{-1}.
$$
$$
=\sum_{\substack{B \subseteq L \\ |B|=n-m-1}} \sum_{\tau \subseteq M}(-1)^{\sigma(B,L)+\sigma(X_{\tau},X)+|\tau|+(k-n-1)(t+1)} P_{(L \setminus B)\cup X_{M \setminus \tau}}(A \cup \{ y \} \cup Y_{\tau}, \{x \} \cup B \cup X_{\tau})
$$
$$
\times
\prod_{z \in \Omega \cup B \cup X_{\tau} \cup Y_{M \setminus \tau} \cup \{ x \}} \det(z,A,y,X_{M\setminus \tau},Y_{\tau}, L\setminus B)^{-1}.
$$
Note that on the right-hand side of the equality we use
$$
\sigma(\{x\} \cup B,\{ x \} \cup L)=\sigma(B,L)+(k-n-1)(t+1)
$$
Rearranging the order of the vectors in the Segre product of the right-hand side (applying Lemma~\ref{numerator} and Lemma~\ref{denominator}) and the vectors in the determinants gives
$$
\sum_{\substack{B \subseteq L \\ |B|=n-m-1}} \sum_{\tau \subseteq M}(-1)^{\sigma(B,L)+\sigma(X_{\tau},X)+|\tau|-|\tau|(t+1)} P_{(L \setminus B)\cup X_{M \setminus \tau} \cup \{ x \}}(A \cup Y_{\tau}, B \cup X_{\tau})
$$
$$
\times
\prod_{z \in \Omega \cup B \cup X_{\tau} \cup Y_{M \setminus \tau} \cup \{ y \}} \det(z,A,X_{M\setminus \tau},x,Y_{\tau}, L\setminus B)^{-1}.
$$
$$
=\sum_{\substack{B \subseteq L \\ |B|=n-m-1}} \sum_{\tau \subseteq M}(-1)^{\sigma(B,L)+\sigma(X_{\tau},X)+|\tau|} P_{(L \setminus B)\cup X_{M \setminus \tau}}(A \cup Y_{\tau}\cup \{ y \} , B \cup X_{\tau}\cup \{x \} )
$$
$$
\times
\prod_{z \in \Omega \cup B \cup X_{\tau} \cup Y_{M \setminus \tau} \cup \{ x \}} \det(z,A,X_{M\setminus \tau},Y_{\tau}, y, L\setminus B)^{-1}.
$$
Finally, note that 
$$\sigma((X \cup \{ x \})_{\tau},X \cup \{ x\})=|\tau|(t+1)+ \sigma(X_{\tau},X)$$
and that 
$$
\sigma((X \cup \{ x \})_{\tau \cup \{ m+1 \}},X \cup \{ x\})=\sigma(X_{\tau},X),$$
from which we deduce that
$$
\sum_{\substack{B \subseteq L \\ |B|=n-m-1}} \sum_{\tau \subseteq M}(-1)^{\sigma(B,L)+\sigma(X^+_{\tau},X^+)+|\tau|} 
 P_{(L \setminus B)\cup X^+_{M^+\setminus \tau} }(A \cup Y^+_{\tau}, B \cup X^+_{\tau})
$$
$$
\times
\prod_{z \in \Omega \cup B \cup X^+_{\tau} \cup Y^+_{M^+ \setminus \tau} } \det(z,A,X^+_{M^+ \setminus \tau},Y^+_{\tau}, L\setminus B)^{-1}.
$$
$$
=\sum_{\substack{B \subseteq L \\ |B|=n-m-1}} \sum_{\tau \subseteq M}(-1)^{\sigma(B,L)+\sigma(X^+_{\tau^+},X^+)+|\tau|} 
P_{(L \setminus B)\cup X^+_{M^+\setminus \tau^+ }}(A \cup Y^+_{\tau^+}, B \cup X^+_{\tau^+}  )
$$
$$
\times
\prod_{z \in \Omega \cup B \cup X^+_{\tau^+ } \cup Y^+_{M^+ \setminus \tau^+}} \det(z,A,X^+_{M^+ \setminus \tau^+},Y^+_{\tau^+}, L\setminus B)^{-1},
$$
where $X^+=X \cup \{ x \}$, $Y^+=Y \cup \{ y\}$, $\tau^+=\tau \cup \{ m+1\}$ and $M^+=M \cup \{ m+1 \}$, which is what we wanted to prove.
\end{proof}

\section{The main theorem}

The following follows from Laplace's formula for determinants.

\begin{lemma} \label{laplace}
Suppose that $W \cup L$ is a basis of ${\mathbb F}_q^k$ and $|X|=n$ and $W=\{ w_1.w_2,\ldots,w_{n+1} \}$. Then
$$
\sum_{j=1}^{n+1} (-1)^{j-1} \det (y,W \setminus w_j,L)\det(w_j,X,L)=\det(W,L)\det(y,X,L).
$$
\end{lemma}

\begin{theorem}
If $q$ is non-prime and $k \leq 2p-2$ then $|S| \leq q+1$.
\end{theorem}

\begin{proof}
By Theorem~\ref{prime}, we can restrict ourselves to the cases $k \geq p+1$.

Suppose $|S|=q+2$ and apply Lemma~\ref{twotothen} with $n=m=k-p$. Then
$$
0=\sum_{\tau \subseteq \{1,\ldots n \}} (-1)^{|\tau|+\sigma(X_{\tau},X)} P_{L \cup X_{M \setminus \tau}} (Y_{\tau},X_{\tau}) \prod_{z \in \Omega \cup X_{\tau} \cup Y_{M \setminus \tau}} \det(z,X_{M \setminus \tau},Y_{\tau},L)^{-1}
$$
where $|L|=p-1$, $\Omega=p-2$ and $|M|=k-p$.

Let $W=\{w_1,w_2,\ldots,w_{2n} \}$ be a subsequence of $S$ disjoint from $L\cup X \cup Y \cup E$, where $E$ is a subset of $\Omega$ of size $p-2-n=2p-k-2$. Define $W_j=\{w_1,w_2,\ldots,w_j \}$.

We shall prove the following by induction on $r \leq n$,
$$
0=\sum_{\tau \subseteq \{1,\ldots n \}} (-1)^{|\tau|+\sigma(X_{\tau},X)} P_{L \cup X_{M \setminus \tau}} (Y_{\tau},X_{\tau}) \prod_{i=1}^r \det(y_{n+1-i},X_{M \setminus \tau},Y_{\tau},L)
$$
$$
 \prod_{z \in E \cup X_{\tau} \cup Y_{M \setminus \tau} \cup W_{n+r}} \det(z,X_{M \setminus \tau},Y_{\tau},L)^{-1}.
$$
For $r=0$ this is the above with $\Omega=E \cup W_n$. Applying the inductive step with $W_{n+r-1}=W_{r+n} \setminus \{ w_j \}$, where $j \in \{r,r+1,\ldots,r+n \}$, we have
$$
0=\sum_{\tau \subseteq \{1,\ldots n \}} (-1)^{|\tau|+\sigma(X_{\tau},X)} P_{L \cup X_{M \setminus \tau}} (Y_{\tau},X_{\tau}) \prod_{i=1}^{r-1} \det(y_{n+1-i},X_{M \setminus \tau},Y_{\tau},L)
$$
$$
\det(w_j,X_{M \setminus \tau},Y_{\tau},L) \prod_{z \in E \cup X_{\tau} \cup Y_{M \setminus \tau} \cup W_{n+r}} \det(z,X_{M \setminus \tau},Y_{\tau},L)^{-1}.
$$
Multiplying by $(-1)^{j-1}\det(y_{n+1-r},W_{n+r} \setminus (W_{r-1} \cup \{w_j \}),L)$, summing over $j\in \{r,r+1,\ldots,r+n \}$ and applying Lemma~\ref{laplace} proves the induction.

For $r=n$ every term in the sum is zero apart from the term corresponding to $\tau=\emptyset$, which gives
$$
0=\prod_{i=1}^n \det(y_{n+1-i},X,L)
 \prod_{z \in E \cup Y \cup W_{2n}} \det(z,X,L)^{-1},
$$
which is a contradiction.
\end{proof}

\begin{cor}
If $q$ is non-prime and $q-2p+4 \leq k \leq q$ then $|S| \leq q+1$.
\end{cor}

\begin{proof}
Suppose that $|S|=q+2$. Then by \cite[Lemma 5.1]{Ball2011} we can construct a set of vectors $S'$ of ${\mathbb F}_q^{q+2-k}$ of size $q+2$ with the property that every subset of $S'$ of size $q+2-k$ is a basis of ${\mathbb F}_q^{q+2-k}$.
\end{proof}

\section{Appendix}

Using the Segre product and the lemmas from Section~\ref{lemmas} we can give a short proof of \cite[Lemma 4.1]{Ball2011}, the main tool used to prove that $|S| \leq q+1$ and classify the case $|S|=q+1$, for $k \leq p$, in \cite{Ball2011}.

\begin{lemma}
Let $L$ of size $r$, $D$ of size $k-1-r$ and $\Omega$ of size $t+2$ be  pairwise disjoint subsequences of $S$. If $1 \leq r \leq t+2$ and $r\leq p-1$, where $q=p^h$, then
$$
0=\sum_{\substack{\Delta \subseteq \Omega \\ |\Delta|=r}}  P_D(\Delta,L) \prod_{z \in (\Omega \setminus \Delta)\cup (L \setminus \ell_0)} \det(z,\Delta,D)^{-1},
$$
where $\ell_0$ is the first element of $L$.
\end{lemma}

\begin{proof}
By induction on $r$. The case $r=1$ follows by dividing the equation in Lemma~\ref{interpolation}, with $E=\Omega$ and $Y=D$, by $T_D(\ell_0)$.

Fix $x \in L$ and apply the induction step to $L\setminus \{x\}$ and $\{x\} \cup D$,
$$
0=\sum_{\substack{\Delta \subseteq \Omega \\ |\Delta|=r-1}}  P_{D\cup \{x\}}(\Delta,L \setminus \{x\}) \prod_{z \in (\Omega \setminus \Delta)\cup (L \setminus \{ \ell_0,x \})} \det(z,\Delta,x,D)^{-1}.
$$
Let $\Delta$ be a subset of $\Omega$ of size $r-1$. Applying Lemma~\ref{interpolation} with $E=(\Omega \cup L)\setminus (\Delta \cup \{\ell_0\})$ and $Y=\Delta \cup D$, we get
$$
0=\sum_{x \in L \setminus \{\ell_0\}} T_{D \cup \Delta}(x) \prod_{z \in (\Omega \setminus \Delta) \cup (L \setminus \{\ell_0,x\})} \det (z,\Delta,x,D)^{-1}$$
$$
+ \sum_{y \in \Omega \setminus \Delta} T_{D \cup \Delta}(y) \prod_{z \in (\Omega \setminus (\Delta \cup \{y\})) \cup (L \setminus \{\ell_0\})} \det(z,\Delta,y,D)^{-1}.
$$
Multiplying by $P_D(\Delta \cup d,L)T_{D \cup \Delta}(d)^{-1}$
for some $d$ for which $T_{D \cup A \cup \Delta}(d) \neq 0$.
By Lemma~\ref{denominator} we can rearrange $L$ so that the last element is $x$, which changes the sign by $\sigma(x,L)$. This gives
$$
0=\sum_{x \in L \setminus \{\ell_0\}} (-1)^{\sigma(x,L)} P_{D \cup \{x\}}( \Delta,L\setminus \{x\}) \prod_{z \in (\Omega \setminus \Delta) \cup (L \setminus \{\ell_0,x\})} \det (z,\Delta,x,D)^{-1}+
$$
$$
 \sum_{y \in \Omega \setminus \Delta} P_{D}(\Delta \cup \{y\},L) \prod_{z \in (\Omega \setminus (\Delta \cup \{y\})) \cup L\setminus \{\ell_0\}} \det(z,\Delta,y,D)^{-1},
$$
since 
$$
P_D(\Delta \cup \{d\},L)T_{D \cup \Delta}(x)T_{D  \cup \Delta}(d)^{-1}=P_{D\cup \{x\}}( \Delta,L \setminus \{x\})
$$
and by Lemma~\ref{switch} (and Lemma~\ref{numerator})
$$
P_D(\Delta \cup \{d\},L)T_{D \cup \Delta}(y) T_{D  \cup \Delta} (d)^{-1}=P_D( \Delta \cup \{y\},L).
$$

Note that in the second term we can order $\Delta \cup \{y\}$ in any way we please without changing the sign since, by Lemma~\ref{numerator}, interchanging two elements of $\Delta \cup \{y\}$ in $P_D(\Delta \cup \{y\},L)$ changes the sign by $(-1)^{t+1}$, exactly the same change occurs when we interchange the same vectors in the product of determinants.

Therefore, when we sum this equation over subsets $\Delta$ of $\Omega$ of size $r-1$ and apply the induction hypothesis, the first sum is zero and the second sum gives
$$
0=r\sum_{\substack{\Delta \subseteq \Omega \\ |\Delta|=r}}  P_D(\Delta,L) \prod_{z \in (\Omega \setminus \Delta)\cup (L \setminus \{ \ell_0\})} \det(z,\Delta,D)^{-1}.
$$

\end{proof}

\bigskip

{\small Simeon Ball}  \\
{\small Departament de Matem\`atica Aplicada IV}, \\
{\small Universitat Polit\`ecnica de Catalunya, Jordi Girona 1-3},
{\small M\`odul C3, Campus Nord,}\\
{\small 08034 Barcelona, Spain} \\
{\small simeon@ma4.upc.edu }

\bigskip

{\small Jan De Beule} \\
{\small Department of Mathematics}, \\
{\small Ghent University},
{\small Krijgslaan 281}, \\
{\small 9000 Gent, Belgium}\\
{\small jdebeule@cage.ugent.be}

\end{document}